\documentclass[a4paper,12pt,oneside,reqno]{amsart} 

\usepackage[english]{babel}
\usepackage[latin9]{inputenc}
\usepackage{amsmath,amsthm,amssymb,amsfonts,amscd,amsbsy,amsxtra}
\usepackage{ucs}
\usepackage{graphicx}
\usepackage{paralist}
\usepackage{hyperref}

\newcommand{\R}{\mathbb{R}}

\newcommand{\N}{\mathbb{N}}

\newcommand{\Cci}{C_{c}^{\infty}}
\newcommand{\pq}{\frac{1}{p}+\frac{1}{q}=1}

\newcommand{\ot}{\frac{1}{t}}
\newcommand{\xot}[1]{\frac{#1}{t}}
\newcommand{\ssubset}{\subset\joinrel\subset}

\DeclareMathOperator{\sign}{sign}

\theoremstyle{plain}
\newtheorem{theorem}{Theorem}[section]
\newtheorem{proposition}[theorem]{Proposition}
\newtheorem{corollary}[theorem]{Corollary}
\newtheorem{lemma}[theorem]{Lemma}
\theoremstyle{definition}
\newtheorem{example}[theorem]{Example}
\newtheorem{remark}[theorem]{Remark}

\title[Mapping Theorems for Sobolev Spaces]{Mapping Theorems for Sobolev Spaces of  Vector-Valued Functions}
\author{\sc W. Arendt}
\address{Wolfgang Arendt\\Institute of Applied Analysis\\Ulm University\\89069 Ulm\\Germany}
\email{wolfgang.arendt@uni-ulm.de}
 
\author{\sc M. Kreuter}
\address{Marcel Kreuter\\Institute of Applied Analysis\\Ulm University\\89069 Ulm\\Germany}
\email{marcel.kreuter@uni-ulm.de}

\date{\today}


\begin{document}
\maketitle

\begin{abstract}
We consider Sobolev spaces with values in Banach spaces as they are frequently useful in applied problems. Given two Banach spaces $X\neq\{0\}$ and $Y$, each Lipschitz continuous mapping $F:X\rightarrow Y$ gives rise to a mapping $u\mapsto F\circ u$ from $W^{1,p}(\Omega,X)$ to $W^{1,p}(\Omega,Y)$ if and only if $Y$ has the Radon-Nikodym Property. But if $F$ is one-sided Gateaux differentiable no condition on the space is needed. We also study when weak properties in the sense of duality imply strong properties. Our results are applied to prove embedding theorems, a multi-dimensional version of the Aubin-Lions Lemma and characterizations of the space $W^{1,p}_0(\Omega,X)$.
\end{abstract}

\section{Introduction}
Sobolev spaces with values in a Banach space are quite natural objects and occur frequently while treating partial differential equations (see e.g. \cite{Amann1995}, \cite{Amann}, \cite{DenkHieberPruess}, \cite{Mugnolo}, \cite{ArendtDier}) and also in probability (see e.g. the upcoming monograph by Hyt\"onen, van Neerven, Veraar and Weis \cite{VanNeerven}). In the case of one space variable quite a few results are known and well documented (see e.g. \cite{CazenaveHaraux} or \cite{Showalter}). However, in higher dimension there seems to be no systematic treatment. The purpose of this article is to study the space $W^{1,p}(\Omega,X)$ where $X$ is a Banach space. We are interested in properties the working analyst frequently needs. In some cases scalar proofs just go through and we will only give a reference. But frequently new proofs and ideas are needed. This is in particular the case for \emph{mapping properties}, our main subject on the article.\\

One question we treat is when each Lipschitz map $F:X\rightarrow Y$ leads to a composition mapping $u\mapsto F\circ u:\ W^{1,p}(\Omega,X)\rightarrow W^{1,p}(\Omega,Y)$. It turns out that in general, this is equivalent to $Y$ having the Radon-Nikodym Property. However, if we merely consider those Lipschitz continuous $F$ that are one-sided Gateaux differentiable, then the result is always true and even a chain rule can be proved. An important case is $F(x)=\|x\|_X$. In this case the chain rule becomes particularly important since it can be used to lift results from the scalar-valued to the vector-valued case. Special attention is also given to the mapping $F(x)=|x|$ where $X$ is a Banach lattice, e.g. $X=L^r(\Omega)$ or $X=C(K)$ which occured in \cite{ArendtDier}. This mapping and more precisely differentiability properties of the projection onto a convex set in Hilbert space have also been studied by Haraux \cite{Haraux}.\\

Weak properties in the sense of duality also form an important subject concerning mapping properties. It is the inverse question we ask: Let $u:\Omega\rightarrow X$ be a function such that $x'\circ u$ has some regularity property for all $x'\in X'$. Does it follow that $u$ has the corresponding regularity property? In other words we ask whether \emph{weak implies strong}. For example, it is well known that a weakly holomorphic map is holomorphic, see \cite{Grothendieck}, \cite[Theorem A.3]{ArendtBatty}. The same is true for harmonic maps, see \cite{ArendtHarmonic}. Here we show that each weakly H\"older continuous function is weakly H\"older continuous. However, in the setting of Sobolev spaces there exist functions $u:\Omega\rightarrow X$ that are not in $W^{1,p}(\Omega,X)$ such that $x'\circ u\in W^{1,p}(\Omega,\R)$ for all $x'\in X'$. On the other hand we show that a Banach lattice $X$ is weakly sequentially complete if and only if each function $u:\Omega\rightarrow X$ such that $x'\circ u\in C^1(\overline{\Omega},\R)$ for all $x'\in X'$ is already in $W^{1,p}(\Omega,X)$. Another positive result concerns Dirichlet boundary conditions: If $u\in W^{1,p}(\Omega,X)$ such that $x'\circ u\in W^{1,p}_0(\Omega,\R)$ for all $x'$ in a separating subset of $X'$ then $u\in W^{1,p}_0(\Omega,X)$. Alternatively we show that $u\in W^{1,p}_0(\Omega,X)$ if and only if $\|u\|_X\in W^{1,p}_0(\Omega,\R)$.\\

The paper is organized as follows. We start by investigating when the quotient criterion characterizes the spaces $W^{1,p}(\Omega,X)$. In Section \ref{section Lipschitz} composition by Lipschitz maps is studied and in Section \ref{gateaux} we add the hypothesis of one-sided Gateaux differentiability. Weak properties and H\"older continuity are studied in Section \ref{section holder}. Finally we apply our results to investigate embedding theorems in Sections \ref{applications} and \ref{AubinLionsApplication} and weak boundary data in Section \ref{Dirichlet}.\\

We will denote all norms by $\|\cdot\|_X$, where $X$ denotes the Banach space to which the norm belongs, and operator norms by $\|\cdot\|_\mathcal{L}$. Also $B_X(x,r)$ and $S_X(x,r)$ will denote the open ball and the sphere of radius $r$ centered at $x\in X$ respectively.

\section{The Difference Quotient Criterion and the Radon-Nikodym Property}\label{section RNP}
Let $\Omega\subset\R^d$ be open, $X$ a Banach space and $1\leq p\leq\infty$. As in the real-valued case, the first Sobolev space $W^{1,p}(\Omega,X)$ is the space of all functions $u\in L^p(\Omega,X)$ for which there exist functions $D_j u\in L^p(\Omega,X)$ $(j=1,\ldots,d)$ such that
\begin{align*}
\int_{\Omega}{u\,\partial_j\varphi}=-\int_{\Omega}{D_j u\,\varphi}
\end{align*}
holds for all $\varphi\in\Cci(\Omega,\R)$, where we use the notation $\partial_j\varphi:=\frac{\partial}{\partial x_j}\varphi$ for the classical partial derivative. The function $D_ju$ is called the \emph{distributional derivative} of $u$ in direction $j$ (we refrain from using the common term 'weak derivative', since that may be confused with differentiability in the weak topology of $X$). Analogously to the real-valued case one sees that $D_ju$ is unique and that $D_ju=\partial_j u$ if $u\in C^1(\Omega,X)\cap L^p(\Omega,X)$ with $\partial_ju\in L^p(\Omega,X)$. The space $W^{1,p}(\Omega,X)$, equipped with the norm $\|u\|_{W^{1,p}(\Omega,X)} := \|u\|_{L^p(\Omega,X)}+\sum_j{\|D_ju\|_{L^p(\Omega,X)}}$ is a Banach space.\\

We want to establish a criterion for a function $u\in L^p(\Omega,X)$ to be in $W^{1,p}(\Omega,X)$ which is well known if $X=\R$. We start with a look at the following property of functions in $W^{1,p}(\Omega,X)$:

\begin{lemma}\label{easydiffquotient}
If $u\in W^{1,p}(\Omega,X)$ with $1\leq p\leq\infty$, then there exists a constant $C$ such that for all $\omega\subset\joinrel\subset\Omega$ and all $h\in\R$ with \mbox{$|h|<\text{dist}(\omega,\partial\Omega)$} we have
\begin{align}\label{criterion}
\|u(\cdot+h\,e_j)-u\|_{L^p(\omega,X)}\leq C|h|\quad(j=1,\ldots,d).
\end{align}
Moreover we can choose $C=\max_{j=1,\ldots,d}\|D_ju\|_{L^p(\Omega,X)}$.
\end{lemma}

\begin{proof}
This can be proven analogously to the real-valued case, see \cite[Proposition 9.3]{Brezis}. Note that the proof is based on regularization arguments which work analogously in the vector-valued case.
\end{proof}

It is well known that a function $u\in L^p(\Omega,\R)$ $(1<p\leq\infty)$ which satisfies criterion (\ref{criterion}) is in $W^{1,p}(\Omega,\R)$, see \cite[Proposition 9.3]{Brezis} and that this is false in general if $p=1$, see \cite[Chapter 9, Remark 6]{Brezis}. We will refer to (\ref{criterion}) as the \emph{Difference Quotient Criterion}. We are interested in extending this criterion to Banach spaces. Such theorems have been proven in special cases, e.g. for $p=2$ and $X$ a Hilbert space in \cite[Lemma A.3]{Mugnolo}. We will show that the criterion describes the spaces $W^{1,p}(\Omega,X)$ if and only if $X$ has the \emph{Radon-Nikodym Property}, that is, every Lipschitz continuous function $f$ from an interval $I$ to $X$ is differentiable almost everywhere, see \cite[Section 1.2]{ArendtBatty}.\\

\begin{theorem}\label{RNP implies criterion}
Let $1<p\leq\infty$ and let $u\in L^p(\Omega,X)$ where $X$ is a Banach space that has the Radon-Nikodym Property. Assume that $u$ satisfies the Difference Quotient Criterion (\ref{criterion}). Then $u\in W^{1,p}(\Omega,X)$ and $\|D_ju\|_{L^p(\Omega,X)}\leq C$ for all $j=1,\ldots,d$.
\end{theorem}

We will use the fact that $L^p(\Omega,X)$ inherites the Radon-Nikodym Property from $X$.

\begin{theorem}[Sundaresan, Turett, Uhl]\label{Sundaresan}
If $(S,\Sigma,\mu)$ is a finite measure space and $1<p<\infty$, then $L^p(S,X)$ has the Radon-Nikodym Property if and only if $X$ does.
\end{theorem}

\begin{proof}
See \cite{TurettUhl}.
\end{proof}

\begin{proof}[Proof of Theorem \ref{RNP implies criterion}]
We first consider the case $1<p<\infty$. Fix a direction $j\in\{1,\ldots,d\}$ and let $\omega\ssubset\Omega$ be bounded. We claim that the distributional derivative of $u_{|\omega}$ exists in $L^p(\omega,X)$ and that its norm is bounded by $C$. For that let $\omega\ssubset\omega'\ssubset\Omega$ and $\tau>0$ be small enough such that the function
\begin{align*}
G:&(-\tau,\tau)\rightarrow L^p(\omega',X)\\
&t\mapsto u(\cdot+te_j)
\end{align*}
is well defined. By assumption $G$ is Lipschitz continuous and hence differentiable almost everywhere by Theorem \ref{Sundaresan}. Fix $0<t_0<\text{dist}(\omega,\partial\omega')$ such that
\begin{align*}
G'(t_0)=\lim_{h\rightarrow0}\frac{u(\cdot+(t_0+h)e_j)-u(\cdot+t_0e_j)}{h}
\end{align*}
exists in $L^p(\omega,X)$. Choose a sequence $h_n\rightarrow0$ such that the above convergence holds almost everywhere in $\omega$, then the function
\begin{align*}
g_\omega(\xi)&:=\lim_{n\rightarrow\infty}\frac{u(\xi+h_ne_j)-u(\xi)}{h_n}\\
&=\lim_{n\rightarrow\infty}\frac{u(\xi-t_0e_j+(t_0+h_n)e_j)-u(\xi-t_0e_j+t_0e_j)}{h_n}
\end{align*}
is an element of $L^p(\omega,X)$ whose norm is bounded by $C$. Given $\varphi\in\Cci(\omega,\R)$ the Dominated Convergence Theorem implies that
\begin{align*}
\int_\omega{\varphi g_\omega}&=\lim_{n\rightarrow\infty}\int_{\omega}{\varphi(\xi)\frac{u(\xi+h_ne_j)-u(\xi)}{h_n}\,d\xi}\\
&=\lim_{n\rightarrow\infty}\int_{\omega}{\frac{\varphi(\xi-h_ne_j)-\varphi(\xi)}{h_n}u(\xi)\,d\xi}\\
&=-\int_\omega{\partial_j\varphi u}.
\end{align*}
This proves the claim. Now let $\omega_n\ssubset\omega_{n+1}\ssubset\Omega$ such that $\bigcup_{n\in\N}{\omega_n}=\Omega$ and let $g_{\omega_n}$ be the corresponding functions found in the first step of the proof. These functions may be pieced together to a function $g\in L^p(\Omega,X)$ whose norm is bounded by $C$. For any $\varphi\in\Cci(\Omega,\R)$ there exists a set $\omega_n$ such that $\varphi\in\Cci(\omega_n,\R)$. Thus the first step shows that $g=D_ju$ finishing the case $1<p<\infty$.\\

For $p=\infty$ let $\omega\ssubset\Omega$ be bounded, then $u\in L^q(\omega,X)$ for all $q<\infty$. Let $\omega_0\ssubset\omega$ and $|h|\leq\text{dist}(\omega_0,\partial\omega)$. Then
\begin{align*}
\|u(\cdot+he_j)-u\|_{L^q(\omega_0,X)}\leq C|h|\lambda(\omega)^{\frac{1}{q}},
\end{align*}
hence $u\in W^{1,q}(\omega,X)$ with $\|D_ju\|_{L^q(\omega,X)}\leq C\lambda(\omega)^{\frac{1}{q}}$ by the first part of the proof. Letting $q\rightarrow\infty$ yields that $\|D_ju\|_{L^\infty(\omega,X)}\leq C$ and hence $u\in W^{1,\infty}(\omega,X)$. The proof can now be finished analogously to the case $1<p<\infty$.
\end{proof}

We now want to show that the Radon-Nikodym Property of $X$ is not only sufficient for the Difference Quotient Criterion to work, but also necessary. The prove this, we need the following result on Sobolev functions in one dimension.

\begin{proposition}\label{fundamental}
Let $I$ be an interval and let $1\leq p\leq\infty$. For a function $u\in L^p(I,X)$ the following are equivalent
\begin{compactenum}[(i)]
\item $u\in W^{1,p}(I,X)$
\item $u$ is absolutely continuous, differentiable almost everywhere and $\frac{d}{dt}u\in L^p(I,X)$
\item There exists a function $v\in L^p(I,X)$ and a $t_0\in I$ such that $u(t)=u(t_0)+\int_{t_0}^{t}{v(s)\,ds}$ holds almost everywhere
\end{compactenum}
\end{proposition}

\begin{proof}
See \cite[Theorem 1.4.35]{CazenaveHaraux}.
\end{proof}

\begin{theorem}\label{RNP criterion other direction}
Let $1<p\leq\infty$. A Banach space $X$ has the Radon-Nikodym Property if and only if the Difference Quotient Criterion characterizes the space $W^{1,p}(\Omega,X)$.
\end{theorem}

\begin{proof}
It remains to show that the Difference Quotient Criterion implies the Radon-Nikodym Property. Let $f:I\rightarrow X$ be Lipschitz continuous. After scaling we may assume that $I^d\ssubset\Omega$. Let $u(\xi):=f(\xi_1)$. If we cut off $u$ outside a compact set in $\Omega$ without changing its values in $I^d$, then $u\in L^p(\Omega,X)$. Hence we may without loss of generality assume that $\Omega=I^d$. Since $f$ is Lipschitz continuous, $u$ satisfies the Difference Quotient Criterion, hence $u\in W^{1,p}(\Omega,X)$. Let $\varphi\in\Cci(I,\R)$ and let $\hat{\varphi}\in\Cci(I^{d-1},\R)$ such that $\int_{I^d}{\hat{\varphi}}=1$. Since $\psi(\xi):=\varphi(\xi_1)\cdot\hat{\varphi}(\xi_2,\ldots,\xi_d)\in\Cci(\Omega,\R)$ we have that
\begin{align*}
\int_{I}{f(t)\varphi'(t)\,dt}=\int_{\Omega}{u(\xi)\partial_1\psi(\xi)\,d\xi}=-\int_{\Omega}{D_1u(\xi)\psi(\xi)\,d\xi}.
\end{align*}
Since $D_1u(\xi)=D_1u(\xi_1)$ is independent of $\xi_2,\ldots,\xi_d$ we obtain that
\begin{align*}
\int_{I}{f(t)\varphi'(t)\,dt}=-\int_{I}{D_1u(t)\varphi(t)\,dt}
\end{align*}
where $D_1u\in L^p(\Omega,X)$. Hence $f\in W^{1,p}(I,X)$ is differentiable almost everywhere by Proposition \ref{fundamental} which finishes the proof.
\end{proof}

\begin{remark}
Hyt{\"o}nen, van Neerven, Veraar and Weis \cite[Proposition 2.81]{VanNeerven} give a different proof of Theorem \ref{RNP criterion other direction} independent of us.
\end{remark}

The difference quotient criterion Theorem \ref{RNP criterion other direction} yields our first mapping theorem.

\begin{corollary}\label{mapping corollary}
Let $1<p\leq\infty$, $\Omega\subset\R^d$ open, $X$ and $Y$ Banach spaces, $Y$ enjoying the Radon-Nikodym Property. Let $F:X\rightarrow Y$ be a Lipschitz continuous mapping such that $F(0)=0$ if $\Omega$ has infinite measure. Then $u\mapsto F\circ u$ defines a mapping
\begin{align*}
W^{1,p}(\Omega,X)\rightarrow W^{1,p}(\Omega,Y).
\end{align*}
\end{corollary}

\section{Composition with Lipschitz Continuous Mappings in Spaces that have the Radon-Nikodym Property}\label{section Lipschitz}

In this section we will give an alternative proof to the last corollary that also includes the case $p=1$. It will also contain another characterization of the Radon-Nikodym Property via Sobolev spaces. The Radon-Nikodym Property can also be formulated as follows.

\begin{lemma}\label{RNP equivalence absolute}
A Banach space $X$ has the Radon-Nikodym Property if and only if every absolutely continuous function $f:I\rightarrow X$ is differentiable almost everywhere.
\end{lemma}

\begin{proof}
See \cite[Proposition 1.2.4]{ArendtBatty}.
\end{proof}

We first consider the one-dimensional case, in which Proposition \ref{fundamental} yields the result right away.

\begin{theorem}\label{Lipschitz1}
Let $1\leq p\leq\infty$ and let $I$ be an interval. Let $X,Y$ be Banach spaces such that $Y$ has the Radon-Nikodym Property and assume that $F:X\rightarrow Y$ is Lipschitz continuous. If $F(0)=0$ or $I$ is bounded then $F\circ u\in W^{1,p}(I,Y)$ for all $u\in W^{1,p}(I,X)$.
\end{theorem}

\begin{proof}
Let $u\in W^{1,p}(I,X)$. Proposition \ref{fundamental} yields that $F\circ u$ is absolutely continuous and hence differentiable almost everywhere by the preceding lemma. By Proposition \ref{fundamental} it remains to show that the derivative of $F\circ u$ is an element of $L^p(I,Y)$. Let $L$ be the Lipschitz constant of $F$, then
\begin{align*}
\|(F\circ u)'(t)\|_X&=\lim_{h\rightarrow0}\frac{\|F(u(t+h))-F(u(t))\|_X}{|h|}\\
&\leq \limsup_{h\rightarrow0} L \frac{\|u(t+h)-u(t)\|_X}{|h|}=L\|u'(t)\|_X
\end{align*}
for almost all $t\in I$, which proves the claim.
\end{proof}

To show Theorem \ref{Lipschitz1} for general domains we will need a higher-dimensional version of Proposition \ref{fundamental}. It was Beppo Levi \cite{Levi} who introduced Sobolev spaces by considering functions which are absolutely continuous on lines. For this reason we allow ourselves to introduce the following terminology:\\

A function $u:\R^d\rightarrow X$ has the \emph{BL-property} if for each $j\in\{1,\ldots,d\}$ for almost all $(x_1,\ldots,x_{j-1},x_{j+1},\ldots,x_d)\in\R^{d-1}$ the function
\begin{align*}
u_j:\,&\R\rightarrow X\\
&t\mapsto u((x_1,\ldots,x_{j-1},t,x_{j+1},\ldots,x_d))
\end{align*}
is absolutely continuous and differentiable almost everywhere. As a consequence of Fubini's Theorem, the partial derivatives $\partial_ju$ of $u$ exist almost everywhere on $\R^d$. Of course, if $X$ has the Radon-Nikodym Property the condition that $u_j$ is differentiable almost everywhere is automatically satisfied if $u_j$ is absolutely continuous.

\begin{theorem}\label{lines}
Let $\Omega\subset\R^d$ be open $1\leq p\leq\infty$.
\begin{compactenum}[(i)]
\item Let $u\in W^{1,p}(\Omega,X)$. Then for each $\omega\ssubset\Omega$ there exists a function $u^*:\R^d\rightarrow X$ which has the BL-property such that $u^*=u$ almost everywhere on $\omega$. Moreover $\partial_ju^*=D_ju$ almost everywhere on $\omega$.
\item Conversely, let $u\in L^p(\Omega,X)$ and $c\geq0$. Assume that for each $\omega\ssubset\Omega$ there exists a function $u^*:\R^d\rightarrow X$ with the BL-property such that $u^*=u$ almost everywhere on $\omega$ and $\|\partial_ju^*\|_{L^p(\Omega,X)}\leq c$ for each $j\in\{1,\ldots,d\}$. Then $u\in W^{1,p}(\Omega,X)$.
\end{compactenum}
\end{theorem}

\begin{proof}
This can be proven analogously to the real-valued case, see \cite[Theorem 1.41]{MalyZiemer}.
\end{proof}

Combining Theorem \ref{lines} with the one-dimensional case we obtain a proof of Corollary \ref{mapping corollary} which also includes the case $p=1$:

\begin{theorem}\label{Lipschitz}
Suppose that $X,Y$ are Banach spaces such that $Y$ has the Radon-Nikodym Property and let $1\leq p\leq\infty$. Let $F:X\rightarrow Y$ be Lipschitz continuous and assume that $F(0)=0$ if $\Omega$ has infinite measure. Then $F\circ u\in W^{1,p}(\Omega,Y)$ for any $u\in W^{1,p}(\Omega,X)$.
\end{theorem}

We give a special example, that will be of interest throughout the rest of this article.

\begin{corollary}\label{norm W1p}
Let $1\leq p\leq\infty$. If $u\in W^{1,p}(\Omega,X)$, then $\|u\|_X\in W^{1,p}(\Omega,\R)$.
\end{corollary}

\begin{remark}
Pelczynski and Wojciechowski have given another proof to Corollary \ref{norm W1p}, see \cite[Theorem 1.1]{PelczynskiWojciechowski}. This important example will play a major role in many results we will prove about the spaces $W^{1,p}(\Omega,X)$. We will extend it in later sections by computing the distributional derivative of $\|u\|_X$ and showing that $\|\cdot\|_X:W^{1,p}(\Omega,X)\rightarrow W^{1,p}(\Omega,\R)$ is continuous.
\end{remark}

As in the last section, we obtain the converse of Theorem \ref{Lipschitz} -- and hence a characterization of the Radon-Nikodym Property.

\begin{theorem}
Let $\Omega$ have finite measure and let $X,Y$ be Banach spaces, $X\neq\{0\}$. If $Y$ has the property that every Lipschitz continuous $F:X\mapsto Y$ gives rise to a mapping $u\mapsto F\circ u$ from $W^{1,p}(\Omega,X)$ to $W^{1,p}(\Omega,Y)$, then $Y$ has the Radon-Nikodym Property.
\end{theorem}

\begin{proof}
Let $f:I\rightarrow Y$ be Lipschitz continuous where without loss of generality $I=[0,1]$. Further choose a vector $x_0\in X$ and a functional $x_0'\in X'$ such that $\langle x_0,x_0'\rangle=1$. Define
\begin{align*}
F:\ &X\rightarrow Y
\end{align*}
by $F(x):=f(\langle x,x_0'\rangle)$. Then $F$ is Lipschitz continuous. We may assume that $I^d\subset\Omega$. Define $\tilde{u}:\Omega\rightarrow X$ via
\begin{align*}
\tilde{u}(\xi):=\xi_1\cdot x_0.
\end{align*}
Choose a function $\varphi\in\Cci(\Omega,\R)$ such that $\varphi_{I^d}\equiv 1$ and let $u:=\varphi\tilde{u}\in W^{1,p}(\Omega,X)$. By assumption we have that $F\circ u\in W^{1,p}(\Omega,Y)$ hence by Theorem \ref{lines} its partial derivative with repect to $\xi_1$ exists for almost all $\xi\in I^d$. But for any such $\xi$ this derivative is equal to $\frac{d}{d\xi_1}f(\xi_1)$, hence $f$ is differentiable almost everywhere on $I$. It follows that $Y$ has the Radon-Nikodym Property.
\end{proof}

\section{Composition with One-Sided Gateaux Differentiable Lipschitz Continuous Mappings}\label{gateaux}

In this section we want to drop the Radon-Nikodym Property. If we do so, then there exist Lipschitz continuous mappings $F$ and distributionally differentiable functions $u$ such that the composition $F\circ u$ is not distributionally differentiable. For this reason we add the assumptions that the mapping $F$ is one-sided Gateaux differentiable and show, that this is sufficient for $F\circ u$ to be distributionally differentiable. We will also prove a chain rule in this setting and explicitly compute the distributional derivatives for the cases that $F$ is a norm or that $F$ is a lattice operation.\\

Let $X,Y$ be Banach spaces. We say that a function $F:X\rightarrow Y$ is \emph{one-sided Gateaux differentiable} at $x$ if the right-hand limit
\begin{align*}
D_{v}^{+}F(x):=\lim_{t\rightarrow0+}\ot(F(x+tv)-F(x))
\end{align*}
exists for every direction $v\in X$. In this case, the left-hand limit
\begin{align*}
D_{v}^{-}F(x):=\lim_{t\rightarrow0+}\frac{1}{-t}(F(x-tv)-F(x))
\end{align*}
exists as well and is given by $D_{v}^{-}F(x)=-D_{-v}^+F(x)$. Let $\Omega\subset\R^d$ be open. For $u:\Omega\rightarrow X$ and $\xi\in\Omega$ we denote by
\begin{align*}
\partial_j^\pm u:=\lim_{t\rightarrow0\pm}\frac{u(\xi+te_j)-u(\xi)}{t}
\end{align*}
the one-sided partial derivatives if they exist.

\begin{lemma}[Chain Rule]\label{chain rule}
Let $X,Y$ be Banach spaces, $u:\Omega\rightarrow X$, $j\in1,\ldots,d$ and $\xi\in\Omega$ such that the partial derivative $\partial_ju(\xi)$ exists. If $F:X\rightarrow Y$ is one-sided Gateaux differentiable, then we have
\begin{align*}
\partial_j^{\pm}F\circ u(\xi)=D_{\partial_ju(\xi)}^{\pm}F(u(\xi)).
\end{align*}
\end{lemma}

\begin{proof}
Let $t>0$, then
\begin{align*}
&\ot\left(F(u(\xi+te_j))-F(u(\xi))\right)\\
=&\ot\left(F(u(\xi+te_j))-F(u(\xi)+t\partial_ju(\xi))\right)\\
&+\ot\left(F(u(\xi)+t\partial_ju(\xi))-F(u(\xi))\right).
\end{align*}
The first expression can be estimated by
\begin{align*}
\xot{L}\|u(\xi+te_j)-u(\xi)-t\partial_ju(\xi)\|_X\rightarrow 0\quad(t\rightarrow0)
\end{align*}
and the second expression converges to $D_{\partial_ju(\xi)}^{+}F(u(\xi))$ as claimed. The left-hand limit can be computed analogously.
\end{proof}

Of course if $F$ is one-sided Gateaux differentiable, in general the right- and left Gateaux derivatives $D_{v}^{+}F(x)$ and $D_{v}^{-}F(x)$ are different. However, if we compose $F$ with a Sobolev function $u$ something surprising happens:

\begin{theorem}\label{Gateaux derivative}
Let $1\leq p\leq\infty$ and $u\in W^{1,p}(\Omega,X)$. Suppose that $F:X\rightarrow Y$ is Lip\-schitz continuous and one-sided Gateaux differentiable and assume furthermore that $\Omega$ is bounded or that $F(0)=0$. Then $F\circ u \in W^{1,p}(\Omega,Y)$ and we have the chain rule
\begin{align*}
D_j(F\circ u)=D^{+}_{D_ju}F(u)=D^{-}_{D_ju}F(u)
\end{align*}
\end{theorem}

We will need the following consequence of Theorem \ref{lines}.

\begin{lemma}\label{partial derivative direction}
Let $1\leq p\leq\infty$, $u\in W^{1,p}(\Omega,X)$ and $j\in\{1,\ldots,d\}$ such that $\partial_j^+u(\xi)$ exists almost everywhere. Then $D_ju=\partial^+_ju$ almost everywhere. The same holds for the left-sided derivative.
\end{lemma}

\begin{proof}
Without loss of generality $j=d$. For $\xi\in\R^d$ we write $\xi=(\hat{\xi},\xi_d)$ with $\hat{\xi}\in\R^{d-1}$ and $\xi_d\in\R$. Let $\omega_n\ssubset\Omega$ such that $\bigcup_{n\in\N}\omega_n=\Omega$. It suffices to show that $D_ju=\partial^+_ju$ almost everywhere on each $\omega_n$. Fix an $n\in\N$ and let $u^*$ be a representative of $u$ on $\omega_n$ as in Theorem \ref{lines}. Choose a null set $N\subset\omega_n$ such that $u^*=u$, the functions $\partial_d^+u$ and $\partial_du^*$ exist and $\partial_du^*=D_du$ on $\omega_n\backslash N$. By Fubini's Theorem the set
\begin{align*}
\omega_n':=\{\xi\in\omega_n\backslash N,(\hat{\xi},\xi_d+t_k)\notin N\textnormal{ for some sequence }t_k\downarrow0\}
\end{align*}
has full measure in $\omega_n$. For $\xi\in\omega_n'$ we choose a suitable sequence and obtain
\begin{align*}
\partial^+_du(\xi)&=\lim_{k\rightarrow\infty}\frac{u(\xi+t_ke_d)-u(\xi)}{t_k}\\
&=\lim_{k\rightarrow\infty}\frac{u^*(\xi+t_ke_d)-u^*(\xi)}{t_k}=\partial_du^*(\xi)=D_du(\xi),
\end{align*}
which finishes the proof.
\end{proof}

\begin{proof}[Proof of Theorem \ref{Gateaux derivative}]
At first note that $F\circ u\in L^p(\Omega,X)$. Fix a $j\in\{1,\ldots,d\}$. The function
\begin{align*}
\Omega&\rightarrow Y\\
\xi&\mapsto D^{\pm}_{D_j(u(\xi))}F(u(\xi))
\end{align*}
is measurable as a limit of a sequence of measurable functions. Since $F$ is Lipschitz continuous, it follows that $D^{\pm}_{D_j(u(\cdot))}F(u(\cdot))\in L^p(\Omega,Y)$. Fix $\varphi\in\Cci(\Omega,\R)$ and let $\omega\ssubset\Omega$ contain its support. We have to show that
\begin{align*}
\int_{\Omega}{F\circ u\ \partial_j\varphi}=-\int_{\Omega}{D^{\pm}_{D_ju(\xi)}F(u(\xi))\varphi(\xi)\,d\xi}.
\end{align*}
Choose a representative $u^*$ of $u$ on $\omega$ as in Theorem \ref{lines}. For every $y'\in Y'$ the function $y'\circ F$ is Lipschitz continuous, hence we have that $\langle F\circ u,y'\rangle\in W^{1,p}(\Omega,\R)$ by Theorem \ref{Lipschitz}. Further $u^*$ is partially differentiable almost everywhere on $\omega$ and $y'\circ F$ is Gateaux differentiable, hence $\partial_j^{\pm}\langle F\circ u^*(\xi),y'\rangle=\langle D^{\pm}_{\partial_ju^*(\xi)}F(u(\xi)),y'\rangle$ almost everywhere on $\omega$ by Lemma \ref{chain rule}. Applying Lemma \ref{partial derivative direction} to this we obtain
\begin{align*}
D_j\langle F\circ u,y'\rangle=\langle D^{\pm}_{D_ju(\cdot)}F(u(\cdot),y'\rangle
\end{align*}
almost everywhere on $\omega$. Thus
\begin{align*}
\left\langle\int_{\Omega}{F\circ u\ \partial_j\varphi},y'\right\rangle=\left\langle-\int_{\Omega}{D^{\pm}_{D_ju(\xi)}F(u(\xi))\varphi(\xi)\,d\xi},y'\right\rangle.
\end{align*}
Since $y'\in X'$ was arbitrary, we obtain the result.
\end{proof}

We give a first example: The function $F:=\|\cdot\|_X:X\rightarrow\R$ is one-sided Gateaux differentiable. The derivatives at $x\in X$ in direction $h\in X$ are given by
\begin{align*}
D_h^+F(x)=\sup \{\langle h,x'\rangle, x'\in J(x)\}
\end{align*}
and
\begin{align*}
D_h^-F(x)=\inf \{\langle h,x'\rangle, x'\in J(x)\},
\end{align*}
where $J(x)=\{x'\in S_{X'}(0,1),\langle x,x'\rangle=\|x\|_X\}$ is the duality map, see \cite[A-II, Remark 2.4]{Nagel}. We already know that $\|u\|_X\in W^{1,p}(\Omega,\R)$ by Corollary \ref{norm W1p}. Here we obtain a formula for the derivative, which will be crucial later.

\begin{theorem}\label{norm weak derivative}
Let $1\leq p\leq\infty$ and $u\in W^{1,p}(\Omega,X)$. Then $\|u\|_X\in W^{1,p}(\Omega,\R)$ and almost everywhere
\begin{align*}
D_j\|u(\xi)\|_X=\langle D_ju(\xi),J(u(\xi))\rangle,
\end{align*}
where by $\langle D_ju(\xi),J(u(\xi))\rangle$ we denote the unique value of $\langle D_ju(\xi),x'\rangle=\langle D_ju(\xi),y'\rangle$ for all $x',y'\in J(u(\xi))$.
\end{theorem}

\begin{example}\label{normexample}
\begin{compactenum}[(a)]
\item Let $X$ be a Hilbert space, then
\begin{align*}
D_j\|u(\xi)\|_X=\begin{cases}
\left\langle D_ju(\xi),\frac{u(\xi)}{\|u(\xi)\|_X}\right\rangle&\text{ if } u(\xi)\neq0,\\
\qquad\quad0&\text{ if } u(\xi)=0.\\
\end{cases}
\end{align*}
In particular: if $X=\R$, then we find the well-known formula $D_j|u|=\sign(u)D_ju$.
\item Let $X=\ell^1$, $u(\xi)=(u_n(\xi))_{n\in\N}\in W^{1,p}(\Omega,X)$. Evaluating at the $n$-th coordinate we obtain that $u_n\in W^{1,p}(\Omega,\R)$ with $D_ju_n(\xi)=(D_ju(\xi))_n$ almost everywhere. It follows that
\begin{align*}
D_j\|u(\xi)\|_{\ell^1}=\sum_{n=1}^{\infty}{\sign(u_n(\xi))D_ju_n(\xi)}.
\end{align*}
\item Let $X=C(K)$ for a compact topological space $K$. For each $\xi\in \Omega$ let $k_{\xi}\in K$ such that $u(\xi)$ has a global extremum in $k_{\xi}$. Then
\begin{align*}
D_j\|u(\xi)\|_{C(K)}=\sign{\left(u(\xi)(k_{\xi})\right)}D_ju(\xi)(k_{\xi}).
\end{align*}
\end{compactenum}
\end{example}

\begin{corollary}\label{norm estimate}
Let $1\leq p\leq\infty$. For all $u\in W^{1,p}(\Omega,X)$ we have $\|u\|_X\in W^{1,p}(\Omega,\R)$ and the estimate $|D_j\|u(\xi)\|_X|\leq\|D_ju(\xi)\|_X$ is valid almost everywhere. In particular, $\|\,\|u\|_X\,\|_{W^{1,p}(\Omega,\R)}\leq \|u\|_{W^{1,p}(\Omega,X)}$.
\end{corollary}

\begin{remark}
In the case $X=\R$ we have equality of these two expressions, see Example \ref{normexample} (a). However if $\text{dim}\,X\geq2$, then there exists no $C>0$ such that $|D_j\|u(\xi)\|_X|\geq C\|D_ju(\xi)\|_X$ for all $u\in W^{1,p}(\Omega,X)$. This can easily be seen by embedding the $\R^2$-valued function $[0,2\pi]\ni t\mapsto(\sin t,\cos t)$ into $W^{1,p}(\Omega,X)$.
\end{remark}

As a second example we will consider lattice operations. Let $E$ be a real Banach lattice. Typical examples are $L^p(\Omega,\R)$ for $1\leq p\leq\infty$ and $C(K)$ for compact $K$. We want to examine whether functions like $u^+$ are distributionally differentiable if $u\in W^{1,p}(\Omega,E)$. Since such lattice operations are Lipschitz continuous, the results of the previous sections imply that lattice operations leave the space $W^{1,p}(\Omega,E)$ invariant if $E$ has the Radon-Nikodym Property and in the case $E=\R$ this property is a fundamental tool for classical Sobolev spaces. However, $W^{1,p}(\Omega,E)$ may not be invariant under lattice operations in general as the following example shows.

\begin{example}
Let $u:(0,1)\rightarrow E=C([0,1],\R)$ be given by
\begin{align*}
u(t)(r)=r-t.
\end{align*}
We have that
\begin{align*}
u(t)=\text{id}_{(0,1)}+\int_0^t{u'(s)\,ds},
\end{align*}
where $u'(t)=-1_{(0,1)}\in L^p((0,1),E)$. Proposition \ref{fundamental} implies that $u\in W^{1,p}((0,1),E)$. The function $u^+$ is given by
\begin{align*}
u^+(t)(r)=
\begin{cases}
0 &\mbox{if } r<t \\ 
r-t & \mbox{if } r\geq t
\end{cases}
\end{align*}
If $u^+$ would be distributionally differentiable, we could evaluate the difference quotient pointwise since point evaluation is an element of $E'$. Hence the only candidate for the distributional derivative of $u^+$ is $(u^+)'(t)=-1_{(t,1)}\notin E$. Thus $u^+$ cannot be distributionally differentiable.\\
\end{example}

We now want to find Banach lattices $E$ for which $W^{1,p}(\Omega,E)$ is invariant under lattice operations but which are not required to have the Radon-Nikodym Property. For the reader's convenience we will summarize the necessary facts and refer to \cite{Nagel} for a short introduction to Banach lattices and to \cite{Schaefer} and \cite{MeyerNieberg} for further information.\\

Let $E$ be a real Banach lattice. A subset $A\subset E$ is called \emph{downwards directed} if for any two elements $x,y\in A$ there exists an element $z\in A$ such that $z\leq x,y$. It is called \emph{lower order bounded} if there exists $y\in E$ such that $y\leq x$ for all $x\in A$. The Banach lattice $E$ is said to have \emph{order continuous norm} if each lower order bounded downwards directed set $A$ converges, that is there exists an $x_0\in E$ such that $\inf_{x\in A}\|x-x_0\|_E=0$. We write $x_0=:\lim A$.\\

Now let $X$ be a Banach space. A function $F:X\rightarrow E$ is called \emph{convex} if
\begin{align*}
F(\lambda x+(1-\lambda)y)\leq\lambda F(x)+(1-\lambda)F(y)
\end{align*}
for all $x,y\in X$ and $0\leq\lambda\leq1$. This is equivalent to saying that $x'\circ F$ is convex for every $x'\in E'_+$. If $F$ is convex the difference quotients
\begin{align*}
t\mapsto\frac{F(x+ty)-F(x)}{t}
\end{align*}
define an increasing function from $\R\backslash\{0\}$ to $E$. Thus if $E$ has order continuous norm, then $F$ is one-sided Gateaux differentiable and $D_y^-F(x)\leq D_y^+F(x)$ for all $x,y\in E$. From Theorem \ref{Gateaux derivative} we deduce the following.

\begin{proposition}\label{convex}
Let $X$ be a Banach space and $u\in W^{1,p}(\Omega,X)$, where $1\leq p\leq\infty$. Let $F:X\rightarrow E$ be Lipschitz continuous and convex where $E$ is a Banach lattice with order continuous norm. If $\Omega$ is bounded or $F(0)=0$, then $F\circ u\in W^{1,p}(\Omega,E)$ with
\begin{align*}
D_j F\circ u=D^{+}_{D_ju}F(u)=D^{-}_{D_ju}F(u).
\end{align*}
\end{proposition}

Next we consider the function $\vartheta:E\rightarrow E$ given by $\vartheta(v)=|v|$ which is convex. We need the following notation. Let $v\in E_+$, then
\begin{align*}
E_v:=\{w\in E,\exists n\in\N:|w|\leq nv\}
\end{align*}
denotes the \emph{ideal generated by $v$}. We set
\begin{align*}
|v|^d:=\{w\in E,|v|\wedge|w|=0\}.
\end{align*}
Assume that $E$ has order continuous norm. Then $E=E_{v^+}\oplus E_{v^-}\oplus|v|^d$. We denote by $P_{v^+},P_{v^-}$ and $P_{|v|^d}$ the projections along this decomposition. Define $\textnormal{sign }v:E\rightarrow E$ by
\begin{align*}
(\textnormal{sign }v)(w):=P_{v^+}w-P_{v^-}w.
\end{align*}

\begin{example}
Let $E=L^p(\Omega)$ with $1\leq p<\infty$ and let $v\in E$. If we set $\Omega_+:=\{\xi\in\Omega,v(\xi)>0\},\Omega_-:=\{\xi\in\Omega,v(\xi)<0\}$ and $\Omega_d:=\{\xi\in\Omega,v(\xi)=0\}$ then $E_{v^+}=\{w\in E,w=0\textnormal{ on } \Omega_+^c\},E_{v^-}=\{w\in E,w=0\textnormal{ on } \Omega_-^c\}$ and $|v|^d=\{w\in E,w=0\textnormal{ on } \Omega_d^c\}$. Hence $P_{v^+}w=1_{\Omega_+}w,P_{v^-}w=1_{\Omega_-}w$ and $P_{|v|^d}w=1_{\Omega_d}w$. We obtain that
\begin{align*}
(\textnormal{sign }v)(w)=\frac{v}{|v|}w\,1_{v\neq0}.
\end{align*}
\end{example}

\begin{proposition}
Assume that $E$ has order continuous norm. Then $\vartheta$ is one-sided Gateaux differentiable and
\begin{align*}
D_w^+\vartheta(v)=(\textnormal{sign }v)w+P_{|v|^d}|w|\\
D_w^-\vartheta(v)=(\textnormal{sign }v)w-P_{|v|^d}|w|
\end{align*}
\end{proposition}

\begin{proof}
This follows from \cite[C-II, Proposition 5.6]{Nagel}.
\end{proof}

With this and Proposition \ref{convex} we obtain

\begin{theorem}\label{lattice sobolev}
Let $E$ be a Banach lattice with order continuous norm and let $1\leq p\leq\infty$. If $u\in W^{1,p}(\Omega,E)$, then $|u|\in W^{1,p}(\Omega,E)$ and
\begin{align*}
D_j|u|=(\textnormal{sign }u)D_ju.
\end{align*}
\end{theorem}

\begin{corollary}\label{stampacchia}
In the setting of Theorem \ref{lattice sobolev} let $u\in W^{1,p}(\Omega,E)$ and $w\in E_+$. Suppose that $|u|\wedge w=0$, then $|D_ju|\wedge w=0$.
\end{corollary}

\begin{corollary}\label{diff plus}
In the setting of Theorem \ref{lattice sobolev} we have that
\begin{align*}
D_ju^+=P_{u^+}D_ju
\end{align*}
\end{corollary}

\begin{proof}
This follows from Theorem \ref{lattice sobolev} and Corollary \ref{stampacchia} since $u^+=\frac{1}{2}(|u|+u)$.
\end{proof}

\begin{example}\label{Dju+ examples}
\begin{compactenum}[(a)]
\item If $X=\R$ we obtain the well known formula
\begin{align*}
D_ju^+=1_{\{u>0\}}D_ju
\end{align*}
\item Let $(S,\Sigma,\mu)$ be a measure space and let $u\in W^{1,p}(\Omega,L^r(S,\R))$ where $1\leq p\leq\infty$ and $1\leq r<\infty$. The norm on $L^r(S,\R)$ is order continuous. A pointwise comparison using Corollary \ref{diff plus} shows that 
\begin{align*}
D_ju^+(\xi)=D_ju(\xi)\cdot 1_{\{s\in S,u(\xi)(s)>0\}}.
\end{align*}
This result was essentially proven directly in \cite[Proposition 4.1]{ArendtDier}.
\end{compactenum}
\end{example}

Finally we remark that Theorem \ref{lattice sobolev} remains true if $E$ is a complex Banach lattice if $\textnormal{sign }u$ is defined properly, see \cite[C-II, Proposition 5.6]{Nagel}.

\section{H\"older Continuity of Distributionally Differentiable Functions}\label{section holder}

For a vector-valued function $u:\Omega\rightarrow X$ it is natural to ask whether a weak regularity property (in the sense of duality) implies the corresponding strong regularity property. For example if $u$ is locally bounded and $\langle u,x'\rangle$ is harmonic for all $x'$ in a separating subset of $X'$ then $u$ itself is harmonic, see \cite[Theorem 5.4]{ArendtHarmonic}. In this section we will show that H\"older continuity can also be tested in such a way while distributional differentiability cannot. We begin with two counterexamples for the latter.

\begin{example}\label{weird example}
\begin{compactenum}[(i)]
\item The function
\begin{align*}
&u:(0,1)\rightarrow L^r((0,1),\R)\quad(1\leq r\leq\infty)\\
&u(t):=1_{(0,t)}
\end{align*}
is nowhere differentiable as one can show by considering its difference quotient, hence it is not in $W^{1,p}(\Omega,L^r((0,1),\R))$. But for each function $x'\in L^{r'}((0,1),\R)$ we have that $\langle u(t),x'\rangle=\int_{0}^{t}{x'(s)\,ds}$ which is in $W^{1,p}((0,1),\R)$ for each $p\leq {r'}$ by Proposition \ref{fundamental}.\\
\item  Let $A\subset(0,1)$ be a set which is not Lebesgue-measurable and consider the Hilbert space $\ell^2(A)$ with orthonormal base $e_t:=(\delta_{t,s})_{s\in A}$. Consider the function
\begin{align*}
&u:(0,1)\rightarrow \ell^2(A)\\
&u(t)=\begin{cases}
0,&t\notin A\\
e_t,&t\in A.
\end{cases}
\end{align*}
Since for each $x'\in \ell^2(A)$ all coordinates except at most countably many are zero, one has $(u,x')=0$ almost everywhere which is in $W^{1,p}((0,1),\R)$. But $u$ is not even measurable.
\end{compactenum}
\end{example}

However if we assume more regularity, we obtain a positive result. The space $X$ is called \emph{weakly sequentially complete} if each weak Cauchy sequence in $X$ has a weak limit. Each reflexive space is weakly sequentially complete but also $L^1(\Omega)$.

\begin{proposition}
Suppose that $X$ is weakly sequentially complete and let $\Omega\subset\R^d$ be open and bounded. Let $u:\Omega\rightarrow X$ such that there exists a representative $u^*$ for which $\langle u^*,x'\rangle\in C^1(\overline{\Omega})$ holds for all $x'\in X'$. Then $u\in W^{1,p}(\Omega,X)$ for all $1\leq p\leq\infty$.
\end{proposition}

\begin{proof}
Since $X$ is weakly sequentially complete there exists $u_j:\Omega\rightarrow X$ such that $\partial_j\langle u^*,x'\rangle=\langle u_j,x'\rangle$ for all $x'\in X'$. The function $u_j$ is weakly continuous and hence measurable by Pettis' Theorem. Further the uniform boundedness principle implies that $u^*,u_j\in L^p(\Omega,X)$. Since the integration by parts formula holds weakly, the Hahn-Banach Theorem implies that $u\in W^{1,p}(\Omega,X)$ and that $u_j=D_ju$.
\end{proof}

Note that the representative $u^*$ has to be fixed before applying the functional $x'$. The proof does not work if we only assume that $\langle u,x'\rangle$ has a representative in $C^1(\Omega,\R)$, cf. Example \ref{weird example} (ii). The result is also false if we drop the hypothesis that $X$ is weakly sequentially complete as the following example shows.

\begin{example}
The function
\begin{align*}
&u:(0,1)\rightarrow c_0\\
&u(t):=\left(\frac{\sin(nt)}{n}\right)
\end{align*}
is nowhere differentiable as the only candidate for the derivative is $u'(t)=(\cos(nt))_{n\in\N}$ which is not in $c_0$ for any $t\in(0,1)$. By Proposition \ref{fundamental} the function $u$ cannot be in $W^{1,p}((0,1),c_0)$. But each $x'\in c_0'$ can be represented as an absolutely convergent series, thus the function $\langle u,x'\rangle$ is in $C^1([0,1],\R)\subset W^{1,p}((0,1),\R)$.
\end{example}

A Banach lattice $E$ is weakly sequentially complete if and only if it does not have $c_0$ as a closed subspace, see \cite[Theorem 1.c.4]{LindenstraussTzafriri}. Hence combining the above proposition and example we obtain

\begin{corollary}
Let $E$ be a Banach lattice and $\Omega\subset\R^d$ open and bounded. Then the following are equivalent:\\
\begin{compactenum}[(i)]
\item Let $u:\Omega\rightarrow X$ such that $\langle u,x'\rangle\in C^1(\overline{\Omega},\R)$ for all $x'\in X'$. Then $u\in W^{1,p}(\Omega,X)$ for all $1\leq p\leq\infty$.
\item $E$ is weakly sequentially complete.
\item $E$ does not contain a closed subspace isomorphic to $c_0$.
\end{compactenum}
\end{corollary}

We now consider H\"older continuity. For $0<\alpha\leq1$ and $\beta>0$ denote by $C^\alpha_\beta(\Omega,X)$ the set of H{\"o}lder continuous functions such that
\begin{align}\label{Holder}
\|f(\xi)-f(\eta)\|_X\leq\beta|\xi-\eta|^\alpha
\end{align}
holds for all $\xi,\eta\in \Omega$. Unlike distributional differentiability, H\"older continuity can be tested.

\begin{proposition}
Let $u:\Omega\rightarrow X$ be a bounded function and assume that for all $x'\in X'$ there exist constants $\alpha,\beta$ such that $\langle u,x'\rangle\in C^\alpha_\beta(\Omega,\R)$. Then there exist $\alpha_0,\beta_0$ such that $u\in C^{\alpha_0}_{\beta_0}(\Omega,X)$.
\end{proposition}

The proposition is a special case of the following theorem which better suits the context of Sobolev spaces.

\begin{theorem}\label{bounded Holder}
Let $u:\Omega\mapsto X$ be measurable and bounded such that $\langle u,x'\rangle$ has a H\"older continuous representative for each $x'\in X'$. Then $u$ has a H\"older continuous representative. 
\end{theorem}

\begin{proof}
Let $X_n':=\left\{x'\in X', \langle u,x'\rangle\textnormal{ has a representative in }C^{\frac{1}{n}}_n(\Omega,\R)\right\}$. Then $X_n'$ is a closed subset of $X'$. Since $u$ is bounded one has $\bigcup_{n\in\N}{X_n'}=X'$. By Baire's Theorem there exists an $n_0\in\N$ such that $X_{n_0}'$ contains a nontrivial closed ball $B:=\overline{B(x_0',r)}$. Let $x'\in X'$ and define $\lambda:=\frac{r}{\|x'\|_{X'}}$. Then we obtain that $x_0',x_0'+\lambda x'\in B$ and hence
\begin{align*}
|\langle u(\xi)-u(\eta),x'\rangle|\leq\frac{2n_0\|x'\|_{X'}}{r}|\xi-\eta|^{\frac{1}{n_0}}
\end{align*}
holds for all $\xi,\eta\in\Omega$ outside of a nullset.

Since $u$ is measurable, we may assume that $X$ is separable. In this case the Banach-Alaoglu Theorem implies that the unit sphere $S_{X'}(0,1)$, endowed with the weak-$*$ topology, is a compact metric space and hence separable. If $(x'_k)_{k\in\N}$ is dense in $(S',w^*)$ it is norming for $X$. Considering only the functionals $x_k'$ in the above computation yields that
\begin{align*}
|\langle u(\xi)-u(\eta),x_k'\rangle|\leq\frac{2n_0}{r}|\xi-\eta|^{\frac{1}{n_0}}
\end{align*}
holds for all $s,t\in\Omega$ outside the countable union of nulsets. Hence taking the supremum over all $k\in\N$ yields that $u$ has a representative in $C^{\frac{1}{n_0}}_{\frac{2n_0}{r}}(\Omega,X)$.
\end{proof}

As an application of this theorem consider the De Giorgi-Nash Theorem: We consider a second-order partial differential operator $L$ with scalar-valued coefficients. The notions of ellipticity and divergence form are as usual, c.f. \cite[Chapter 8]{GilbargTrudinger}.

\begin{corollary}[De Giorgi-Nash]
Let $u\in W^{1,2}(\Omega,X)$ be a weak solution of the equation $Lu=g+D_if^i$ where $L$ is a strictly elliptic differential operator of second order in divergence form with real measurable bounded coefficients, $f^i\in L^q(\Omega,X)$ and $g\in L^{\frac{q}{2}}(\Omega,X)$ for some $q>d$. Then $u$ is locally H\"older continuous.
\end{corollary}

\begin{proof}
This follows from Theorem \ref{bounded Holder} and \cite[Theorem 8.22]{GilbargTrudinger}.
\end{proof}

We may also use Theorem \ref{bounded Holder} to extend Morrey's Embedding Theorem to vector-valued functions. However, if the constants $\alpha,\beta$ in (\ref{Holder}) are controllable, the result of Theorem \ref{bounded Holder} might be easier to prove. We shall do so in the next section.\\

A similar Baire argument as in Theorem \ref{bounded Holder} gives uniform H\"older exponents in some situations.

\begin{proposition}
Let $Y$ be a Banach space and let $T:Y\rightarrow C_b(\Omega,X)$ be linear and continuous. Assume that for each $y\in Y$ the function $Ty$ is H\"older continuous. Then there exist $\alpha_0,\beta_0$ such that $Ty\in C^{\alpha_0}_{\beta_0}(\Omega,X)$.
\end{proposition}

\section{Embedding Theorems}\label{applications}
As an application of Corollary \ref{norm estimate} we prove that the Sobolev-Gagliar\-do-Nirenberg Embedding Theorems carry over from the real-valued to the vector-valued case.\\

\begin{theorem}\label{embedding}
Let $\Omega\subset\R^d$ be open such that we have a continuous embedding $W^{1,p}(\Omega,\R)\hookrightarrow L^r(\Omega,\R)$ for some $1\leq p,r\leq\infty$. Then we have a continuous embedding $W^{1,p}(\Omega,X)\hookrightarrow L^r(\Omega,X)$ for any Banach space $X$ where the norm of the embedding remains the same.
\end{theorem}

\begin{proof}
Let $u\in W^{1,p}(\Omega,X)$ then $\|u\|_X\in W^{1,p}(\Omega,\R)$. By assumption it follows that $\|u\|_X\in L^r(\Omega,\R)$ and hence by definition of the Bochner space we obtain $u\in L^r(\Omega,X)$. Further if $C$ is the norm of the embedding in the real-valued case, we use Corollary \ref{norm estimate} to compute
\begin{align*}
\|u\|_{L^r(\Omega,X)}&=\|\,\|u\|_X\,\|_{L^r(\Omega,\R)}\leq C\|\,\|u\|_X\,\|_{W^{1,p}(\Omega,\R)}\leq C\|u\|_{W^{1,p}(\Omega,X)}.
\end{align*}
On the other hand, the norm of the embedding $W^{1,p}(\Omega,X)\hookrightarrow L^r(\Omega,X)$ cannot be less than the norm of the embedding $W^{1,p}(\Omega,\R)\hookrightarrow L^r(\Omega,\R)$ hence they are equal.
\end{proof}

\begin{theorem}[Morrey's Embedding Theorem]\label{morrey}
Let $\Omega$ be open and suppose that there exist constants $C$ and $\alpha$ such that
\begin{align*}
|u(\xi)-u(\eta)|\leq C\|u\|_{W^{1,p}(\Omega,\R)}|\xi-\eta|^{\alpha}
\end{align*}
for all $u\in W^{1,p}(\Omega,\R)$ and almost all $\xi,\eta\in\Omega$. Then for all $u\in W^{1,p}(\Omega,X)$ we have that
\begin{align*}
\|u(\xi)-u(\eta)\|_X\leq C\|u\|_{W^{1,p}(\Omega,X)}|\xi-\eta|^{\alpha}
\end{align*}
for almost all $\xi,\eta\in\Omega$. In particular $u$ has a H\"older continuous representative.
\end{theorem}

\begin{proof}
Let $x_k'$ be chosen as in the proof of Theorem \ref{bounded Holder} and define $u_k:=\langle u,x_k'\rangle$. As the integral commutes with $x_k'$ one has $u_k\in W^{1,p}(\Omega,\R)$ with derivative $D_j u_k=\langle D_ju,x_k'\rangle$. By assumption we have
\begin{align*}
|u_k(\xi)-u_k(\eta)|\leq C\|u_k\|_{W^{1,p}(\Omega,\R)}|\xi-\eta|^{\alpha},
\end{align*}
for all $k\in\N$ and all $\xi,\eta$ outside a common set of measure zero. Since $|u_k|\leq\|u\|_X$ and $|D_ju_k|\leq \|D_ju\|_X$ we obtain that $\|u_k\|_{W^{1,p}(\Omega,\R)}\leq \|u\|_{W^{1,p}(\Omega,X)}$. Taking the supremum over all $k\in\N$ on the left side we are left with
\begin{align*}
\|u(\xi)-u(\eta)\|_X\leq C\|u\|_{W^{1,p}(\Omega,X)}|\xi-\eta|^{\alpha}
\end{align*}
and hence $u$ has a H\"older continuous representative.
\end{proof}

Next we prove compactness of the Sobolev embedding. Let $\Omega\subset\R^d$ be a bounded open set with Lipschitz boundary. If $\Omega\subset\R$ is an interval, the following result is known as the Aubin-Lions Lemma, see \cite{Aubin}, \cite[Chapter III, Proposition 1.3]{Showalter}. The Aubin-Lions result is very useful in the theory of partial differential equations, see e.g. \cite{ArendtChill}. Many extensions on intervals have been given, see e.g. \cite{Simon}. Amann \cite[Theorem 5.2]{AmannCompact} gives a multidimensional version if the boundary of $\Omega$ is smooth. Here we prove a special case of Amann's result by direct arguments which only require the boundary to be Lipschitz.

\begin{theorem}\label{AubinLions}
Let $\Omega\subset\R^d$ be a bounded open set with Lipschitz boundary. Suppose $X,Y$ are Banach spaces such that $Y$ is compactly embedded in $X$ and let $1\leq p<\infty$. Then the embedding
\begin{align*}
W^{1,p}(\Omega,X)\cap L^p(\Omega,Y)\hookrightarrow L^p(\Omega,X)
\end{align*}
is compact.
\end{theorem}

For the proof we will need some auxillary results.

\begin{lemma}\label{equicontinuity}
Let $\mathcal{F}\subset L^p(\R^d,X)$ be bounded and $\rho\in\Cci(\R^d,\R)$. Then there exists a $c>0$ such that
\begin{align*}
\|\rho\ast f(\xi)-\rho\ast f(\eta)\|_X\leq c|\xi-\eta|
\end{align*}
for all $\xi,\eta\in\R^d$ and $f\in\mathcal{F}$.
\end{lemma}

\begin{proof}
This can be proven analogously to the scalar-valued case, see e.g. the proof of \cite[Corollary 4.28]{Brezis}.
\end{proof}

\begin{lemma}\label{uniform convolution}
Let $\mathcal{F}\subset W^{1,p}(\R^d,X)$ be bounded and $(\rho_n)$ be a mollifier. Then
\begin{align*}
\sup_{f\in\mathcal{F}}\|\rho_n\ast f-f\|_{L^p(\R^d,X)}\rightarrow0
\end{align*}
as $n\rightarrow\infty$.
\end{lemma}

\begin{proof}
From Lemma \ref{easydiffquotient} it follows that there exists a $C>0$ such that for all $h\in\R^d$ we have
\begin{align*}
\|u(\cdot+h)-u\|_{L^p(\R^d,X)}\leq C|h|
\end{align*}
for every $u\in\mathcal{F}$. The result now follows as in the scalar-valued case, see e.g. Step 1 of the proof of \cite[Theorem 4.26]{Brezis}.
\end{proof}

The following theorem is particularly important as its scalar-valued counterpart is frequently used in the theory of Sobolev spaces. In the last section we will use it again. For the notion of \emph{uniform Lipschitz boundary} we refer to \cite[Definition 12.10]{Leoni}. For bounded sets, uniform Lipschitz boundary is the same as Lipschitz boundary.

\begin{theorem}[Extension Theorem]\label{extension theorem}
Let $\Omega$ be an open set with uniform Lipschitz boundary. Then for any $1\leq p<\infty$ there exists a bounded linear operator
\begin{align*}
\mathcal{E}:W^{1,p}(\Omega,X)\rightarrow W^{1,p}(\R^d,X)
\end{align*}
such that $\mathcal{E}(u)_{|\Omega}=u$ for all $u\in W^{1,p}(\Omega,X)$.
\end{theorem}

\begin{proof}
This can be shown analogously to the scalar-valued case, see e.g. \cite[Theorem 12.15]{Leoni}, also cf. \cite[Theorem 2.4.5]{Amann1995}.
\end{proof}

\begin{proof}[Proof of Theorem \ref{AubinLions}]
We have to show that
\begin{align*}
B:=\{f\in W^{1,p}(\Omega,X)\cap L^p(\Omega,Y), \|f\|_{W^{1,p}(\Omega,X)}\leq1, \|f\|_{L^p(\Omega,Y)}\leq1\}
\end{align*}
is precompact in $L^p(\Omega,X)$. Let $\mathcal{E}$ be the extension operator in Theorem \ref{extension theorem}. We will show the theorem in two steps: First we show that for any $\omega\ssubset\Omega$ the set $B_{|\omega}$ is precompact in $L^p(\omega,X)$. For this let $\varepsilon>0$ and use Lemma \ref{uniform convolution} to choose $n_0>\textnormal{dist}(\omega,\partial\Omega)^{-1}$ such that
\begin{align*}
\|\mathcal{E}f-\rho_{n_0}\ast\mathcal{E}f\|_{L^p(\R^d,X)}\leq\varepsilon
\end{align*}
for all $f\in B$. For all $\xi\in\overline{\omega}$ we have
\begin{align*}
\|(\rho_{n_0}\ast\mathcal{E}f)(\xi)\|_Y\leq\|\rho_{n_0}\|_{L^q(\Omega,\R)}\|f\|_{L^p(\Omega,Y)}\leq\|\rho_{n_0}\|_{L^q(\Omega,\R)}=:C
\end{align*}
where $\pq$. Since $Y\hookrightarrow X$ is compact the set
\begin{align*}
K:=\{y\in Y,\|y\|_Y\leq C\}
\end{align*}
is precompact in $X$. Further by Lemma \ref{equicontinuity} the set
\begin{align*}
H:=\{\rho_{n_0}\ast\mathcal{E}f,f\in B\}
\end{align*}
is equicontinuous, thus it is precompact in $C(\overline{\omega},K)$ by the Arzela-Ascoli Theorem. Consequently we find $g_j\in L^p(\omega,X)$ such that $H\subset\bigcup_{j=1}^{m}B_{L^p(\omega,X)}(g_j,\varepsilon)$. Let $f\in B$ and choose $j\in{1,\ldots,m}$ such that $\|\rho_{n_0}\ast\mathcal{E}f-g_j\|_{L^p(\omega,X)}<\varepsilon$. Hence
\begin{align*}
\|f-g_j\|_{L^p(\omega,X)}\leq\|\mathcal{E}f-\rho_{n_0}\ast\mathcal{E}f\|_{L^p(\omega,X)}+\|\rho_{n_0}\ast\mathcal{E}f-g_j\|_{L^p(\omega,X)}<2\varepsilon
\end{align*}
implying that $B_{|\omega}\subset\bigcup_{j=1}^{m}B_{L^p(\omega,X)}(g_j,2\varepsilon)$ is precompact.\\

Using this we now show that $B$ is precompact in $L^p(\Omega,X)$. Let $\varepsilon>0$ and use Lemma \ref{uniform convolution} to choose $n_0\in\N$ such that
\begin{align*}
\|\mathcal{E}f-\rho_{n_0}\ast\mathcal{E}f\|_{L^p(\R^d,X)}\leq\varepsilon
\end{align*}
for all $f\in B$. For all $\xi\in\Omega$ we have
\begin{align*}
\|(\rho_{n_0}\ast\mathcal{E})(\xi)\|_X&\leq \|\rho_{n_0}\|_{L^q(\R^d,\R)}\|\mathcal{E}f\|_{L^p(\R^d,X)}\\
&\leq \|\rho_{n_0}\|_{L^q(\R^d,\R)}\|\mathcal{E}\|_\mathcal{L}=:C
\end{align*}
independently of $f\in B$. Let $\omega\ssubset\Omega$ such that $C|\Omega\backslash\omega|^{\frac{1}{p}}<\varepsilon$, then $\|\rho_{n_0}\ast\mathcal{E}f\|_{L^p(\Omega\backslash\omega,X)}<\varepsilon$ for all $f\in B$ by the above estimate. By the first step there exist $g_j\in L^p(\omega,X)$ such that
\begin{align*}
B_{|\omega}\subset\bigcup_{j=1}^{m}B_{L^p(\omega,X)}(g_j,\varepsilon).
\end{align*}
Define
\begin{align*}
G_j(\xi):=\begin{cases}
g_j(\xi),&\textnormal{ if }\xi\in\omega\\
0,&\textnormal{ otherwise.}
\end{cases}
\end{align*}
Let $f\in B$, then there exists a $j\in\{1,\ldots,m\}$ such that $\|f_{|\omega}-g_j\|_{L^p(\omega,X)}\leq\varepsilon$. Thus
\begin{align*}
\|f-G_j\|_{L^p(\Omega,X)}&\leq\varepsilon+\|f\|_{L^p(\Omega\backslash\omega,X)}\\
&\leq\varepsilon+\|\mathcal{E}f-\rho_{n_0}\ast\mathcal{E}f\|_{L^p(\R^d,X)}+\|\rho_{n_0}\ast\mathcal{E}f\|_{L^p(\Omega\backslash\omega,X)}\\
&\leq3\varepsilon
\end{align*}
proving that $B\subset\bigcup_{j=1}^{m}B_{L^p(\omega,X)}(G_j,3\varepsilon)$. This finishes the proof.
\end{proof}

\section{Weak Dirichlet Boundary Data}\label{Dirichlet}
From now on let $1\leq p<\infty$. As usual we define the space $W^{1,p}_0(\Omega,X)$ as the closure of $\Cci(\Omega,X)$ in the $W^{1,p}(\Omega,X)$-norm. Our aim is to give some characterizations and properties of this space. We will need the following proposition which is of interest on its own.

\begin{proposition}\label{norm continuous}
The mapping $\|\cdot\|_X:W^{1,p}(\Omega,X)\rightarrow W^{1,p}(\Omega,\R)$ is continuous.
\end{proposition}

\begin{proof}
Let $u_n\rightarrow u$ in $W^{1,p}(\Omega,X)$. Without loss of generality we may assume that $u_n\rightarrow u$ and $D_ju_n\rightarrow D_ju$ pointwise almost everywhere and that they are pointwise dominated by a $L^p$-function. If this is not the case, we may apply a subsequence argument since every subsequence of $u_n$ has a subsequence with these properties. It is obvious that $\|u_n\|_X\rightarrow\|u\|_X$ in $L^p(\Omega,\R)$, hence it only remains to show convergence of the derivatives. Since $|D_j\|u_n\|_X|\leq\|D_ju_n\|_X$ by Corollary \ref{norm estimate} the functions $D_j\|u_n\|_X$ are uniformly dominated by an $L^p$-function. To apply the Dominated Convergence Theorem we need to show that they converge almost everywhere to $D_j\|u\|_X$. Let $\xi\in \Omega\backslash N$ where $N\subset\Omega$ is a negligible set such that $u_n$ and $D_ju_n$ converge pointwise outside of $N$. Chose $x_n'\in J(u_n(\xi))$. Since we are working with a countable number of measurable functions, we may assume that $X$ is separable. The Banach-Alaoglu Theorem implies that any subsequence $(x_{n_k}')$ of $(x_n')$ has a subsequence $(x_{n_{k_l}}')$ which converges in the $\sigma(X',X)$-topology to an element $x'\in X'$. One easily sees that $x'\in J(u(\xi))$. Now we use Theorem \ref{norm weak derivative} to deduce that
\begin{align*}
D_j\|u_{n_{k_l}}(\xi)\|_X=\langle D_j u_{n_{k_l}}(\xi),x_{n_{k_l}}'\rangle\rightarrow\langle D_ju(\xi),x'\rangle=D_j\|u(\xi)\|_X.
\end{align*}
We obtain that $D_j\|u_n\|_X\rightarrow D_j\|u\|_X$ pointwise outside of $N$ as $n\rightarrow\infty$ since the subsequences where chosen arbitrarily.
\end{proof}

If $u\in W^{1,p}_0(\Omega,X)$ and $\varphi_n\in\Cci(\Omega,X)$ converges to $u$ then Proposition \ref{norm continuous} implies that $\|\varphi_n\|_X$ converges to $\|u\|_X$, hence $\|u\|_X\in W^{1,p}_0(\Omega,\R)$. The converse is true as well. We will need the following lemmas.

\begin{lemma}
Let $u\in W^{1,p}(\Omega,X)$ and $\psi\in W^{1,p}(\Omega,\R)$ such that $\psi u$, $(D_j\psi)u$ and $\psi D_ju$ are in $L^p(\Omega,\R)$, then $\psi u\in W^{1,p}(\Omega,X)$ with
\begin{align*}
D_j(\psi u)=(D_j\psi)u+\psi D_ju.
\end{align*}
\end{lemma}

\begin{proof}
If $X=\R$ this follows from \cite[Equation (7.18)]{GilbargTrudinger}. Applying this to $\langle u,x'\rangle$ for arbitrary $x'\in X'$ the result follows from the Hahn-Banach Theorem.
\end{proof}

\begin{lemma}\label{quotient rule}
Let $u\in W^{1,p}(\Omega,X)$ and let $\hat{\varphi}\in \Cci(\Omega,\R)_+$. Define $\varphi:=\hat{\varphi}\wedge\|u\|_X$. Then the function
\begin{align*}
v:=\frac{u}{\|u\|_X}\varphi\,1_{u\neq0}
\end{align*}
is in $W^{1,p}(\Omega,X)$ and the calculus rule
\begin{align*}
D_jv=\left(\frac{(D_ju)\|u\|_X-uD_j\|u\|_X}{\|u\|_X^2}\,\varphi+\frac{u}{\|u\|_X}D_j\varphi\right)\,1_{u\neq0}
\end{align*}
holds.
\end{lemma}
\begin{proof}
Note that the functions
\begin{align*}
\frac{u}{\|u\|_X}\varphi\,1_{u\neq0},\ \frac{D_ju}{\|u\|_X}\varphi\,1_{u\neq0}, \frac{u}{\|u\|_X}D_j\varphi\,1_{u\neq0}\text{ and } \frac{uD_j\|u\|_X}{\|u\|_X^2}\varphi\,1_{u\neq0}
\end{align*}
are all in $L^p(\Omega,X)$. Let $\varepsilon>0$ and define
\begin{align*}
f_{\varepsilon}:\R_+\rightarrow\R_+\\
t\mapsto\frac{1}{t+\varepsilon}.
\end{align*}
We have that $f_{\varepsilon}\in C^1(\R_+)$ and that $|f_{\varepsilon}'|$ is bounded by $\frac{1}{\varepsilon^2}$. The usual chain rule, see e.g. \cite[Proposition 9.5]{Brezis}, yields that $f_{\varepsilon}\circ\|u\|_X\in W^{1,p}(\Omega,\R)$ and
\begin{align*}
D_j(f_{\varepsilon}\circ\|u\|_X)=\frac{-1}{(\|u\|_X+\varepsilon)^2}D_j\|u\|_X.
\end{align*}
The preceding lemma implies that $\varphi u\in W^{1,p}(\Omega,X)$ and that the usual product rule holds. Using the lemma once more we obtain that $(f_{\varepsilon}\circ\|u\|_X)\varphi u\in W^{1,p}(\Omega,X)$ and  the usual product rule holds. This means that

\begin{align*}
&\int_{\Omega}{(f_{\varepsilon}\circ\|u\|_X)\varphi u\partial_j\psi}\\
&=-\int_{\Omega}{\varphi\, D_jf_{\varepsilon}\circ\|u\|_Xu\psi}-\int_{\Omega}{D_j(\varphi u)f_{\varepsilon}\circ\|u\|_X\psi}
\end{align*}
for all $\psi\in\Cci(\Omega,\R)$. Letting $\varepsilon\rightarrow0$ by the Dominated Convergence Theorem we obtain
\begin{align*}
&\int_{\Omega}{\varphi\frac{1}{\|u\|_X}u1_{u\neq0}\partial_j\psi}\\
&=\int_{\Omega}{\varphi\, \frac{1}{\|u\|_X^2}D_j\|u\|_Xu\psi}-\int_{\Omega}{D_j(\varphi u)\frac{1}{\|u\|_X}\psi}.
\end{align*}
This proves the claim.
\end{proof}

\begin{theorem}\label{W1p0 norm}
Let $u\in W^{1,p}(\Omega,X)$, then $u\in W^{1,p}_0(\Omega,X)$ if and only if $\|u\|_X\in W^{1,p}_0(\Omega,\R)$
\end{theorem}

\begin{proof}
It remains to show the 'if' part. Let $\hat{\varphi}_n \in \Cci(\Omega,\R)_+$ be convergent to $\|u\|_X$ in $W^{1,p}(\Omega,X)_+$ and define $\varphi_n:=\hat{\varphi}_n\wedge\|u\|_X$. The function $u_n:=\frac{u}{\|u\|_X}\varphi_n\,1_{u\neq0}$ is in $W^{1,p}(\Omega,X)$ by Lemma \ref{quotient rule} and is compactly supported. Using standard convolution arguments one shows that $u_n\in W^{1,p}_0(\Omega,X)$. The calculus rules in Lemma \ref{quotient rule} and the Dominated Convergence Theorem imply that $u_n\rightarrow u$ in $W^{1,p}(\Omega,X)$ and thus $u\in W^{1,p}_0(\Omega,X)$.
\end{proof}

We obtain several corollaries. The first one is an extension of the ideal property of $W^{1,p}_0(\Omega,\R)$.

\begin{corollary}
Let $X,Y$ be Banach spaces. If $u\in W^{1,p}_0(\Omega,X)$ and $v\in W^{1,p}(\Omega,Y)$ such that $\|v\|_Y\leq\|u\|_X$ almost everywhere, then  $v\in W^{1,p}_0(\Omega,Y)$.
\end{corollary}

Next we extend a well known inequality to the vector-valued case.

\begin{corollary}[Poincar\'{e} inequality]
If $\Omega$ is bounded in direction $e_j$ then there exists a constant $C>0$ such that
\begin{align*}
\|D_ju\|_{L^p(\Omega,X)}\geq C\|u\|_{L^p(\Omega,X)}
\end{align*}
for all $u\in W^{1,p}_0(\Omega,X)$. Further this constant conicides with the Po\-in\-ca\-r\'{e} constant in the real case.
\end{corollary}

\begin{proof}
If $u\in W^{1,p}_0(\Omega,X)$, then $\|u\|_X\in W^{1,p}_0(\Omega,\R)$ and hence $\|u\|_X$ satisfies the Poincar\'{e} inequality. Using that $\|u\|_X$ and $u$ have the same $L^p$-norms and that, by Corollary \ref{norm estimate}, $|D_j\|u\|_X|\leq \|D_ju\|_X$ we obtain the result.
\end{proof}

We can also characterize $W^{1,p}_0(\Omega,X)$ weakly. For that we need the following lemma which immediately follows from Theorem \ref{W1p0 norm}.

\begin{lemma}
Let $X$ be a closed subspace of the Banach space $Y$ and let $u\in W^{1,p}(\Omega,X)$. Then $u\in W^{1,p}_0(\Omega,X)$ if and only if $u\in W^{1,p}_0(\Omega,Y)$
\end{lemma}

\begin{theorem}
Let $u\in W^{1,p}(\Omega,X)$. The following are equivalent:
\begin{compactenum}[(i)]
\item $u\in W^{1,p}_0(\Omega,X)$
\item $\langle u,x'\rangle\in W^{1,p}_0(\Omega,\R)$ for every $x'\in X'$
\item $\langle u,x'\rangle\in W^{1,p}_0(\Omega,\R)$ for every $x'$ in a separating subset of $X'$
\end{compactenum}
\end{theorem}

\begin{proof}
$(i)\Rightarrow(ii)$ We have $|\langle u,x'\rangle|\leq\|u\|_X\|x'\|_{X'}\in W^{1,p}_0(\Omega,\R)$ and hence $\langle u,x'\rangle\in W^{1,p}_0(\Omega,\R)$ by the ideal property in the real case.\\

$(ii)\Rightarrow(i)$ Since $u$ is measurable and any functional on a closed subspace of $X$ can be extended to the whole space by the Hahn-Banach Theorem we may assume that $X$ is separable. The Banach-Mazur theorem implies that $X\hookrightarrow C[0,1]$ isometrically. In view of the preceding lemma we may assume that $X=C[0,1]$. Now $X$ has a Schauder basis $(b_k)_{k\in\N}$ with coordinate functionals $(b_k')$. Let $u_n:=P_nu=\sum_{k=1}^{n}{\langle u,b_k'\rangle b_k}$, then $u_n\in W^{1,p}_0(\Omega,X)$ by assumption. Further $u_n\rightarrow u$ pointwise and if $C$ is the basis constant of $(b_k)$ we have that $u_n$ is bounded by $C\|u\|_X$. Hence $u_n\rightarrow u$ in $L^p(\Omega,X)$. Analogously one computes that $D_ju_n\rightarrow D_ju$ in $L^p(\Omega,X)$ and hence $u\in W^{1,p}_0(\Omega,X)$.\\

$(ii)\Rightarrow(iii)$ trivial\\

$(iii)\Rightarrow(ii)$ Consider the space
\begin{align*}
V:=\{x'\in X',\langle u,x'\rangle\in W^{1,p}_0(\Omega,\R)\}
\end{align*}
which by assumption is $\sigma(X',X)$-dense in $X'$. We will show that it is also closed in the $\sigma(X',X)$-topology. By the Krein-Smulyan Theorem, see e.g. \cite[Theorem 2.7.11]{Megginson}, it suffices to show that $V\cap B_{X'}(0,1)$ is $\sigma(X',X)$-closed. As in $(ii)$ we may assume that $X$ is separable and hence the $\sigma(X',X)$-topology is metrizable. Let $x_n'\in V\cap B_{X'}(0,1)$ such that $x_n'\rightharpoonup^*x'$. Then $\langle u,x_n'\rangle\rightarrow\langle u,x'\rangle$ pointwise and the functions $\langle u,x_n'\rangle$ are dominated by $\|u\|_X$. Hence $\langle u,x_n'\rangle\rightarrow\langle u,x'\rangle$ in $L^p(\Omega,\R)$ by the Dominated Convergence Theorem. The same argument shows that $\langle u,x_n'\rangle\rightarrow\langle u,x'\rangle$ in $W^{1,p}(\Omega,\R)$. Since $\langle u,x_n'\rangle\in W^{1,p}_0(\Omega,\R)$ it follows that $\langle u,x'\rangle\in W^{1,p}_0(\Omega,\R)$. Thus $V$ is $\sigma(X',X)$-closed which implies $(ii)$.
\end{proof}

\begin{example}
\begin{compactenum}[(i)]
\item Let $X=C(K)$ for a compact set $K$ and let $u\in W^{1,p}(\Omega,X)$. Then $u\in W^{1,p}_0(\Omega,X)$ if and only if $u(\cdot)(k)\in W^{1,p}_0(\Omega,\R)$ for every $k$ in a dense subset of $K$.
\item Let $X=\ell^r$ for $1\leq r\leq\infty$ and let $u=(u_n)_{n\in\N}\in W^{1,p}(\Omega,X)$. Then $u\in W^{1,p}_0(\Omega,X)$ if and only if $u_n\in W^{1,p}_0(\Omega,\R)$ for all $n\in\N$.
\item Let $X=L^r(S,d\mu)$ for $1\leq r\leq\infty$ and a measure space $S$ and let $u\in W^{1,p}(\Omega,X)$. Then $u\in W^{1,p}(\Omega,X)$ if and only if $\int_T{u(\cdot)(s)\,d\mu(s)}\in W^{1,p}_0(\Omega,\R)$ for all measurable sets $T\subset S$.
\end{compactenum}
\end{example}

Finally we describe $W^{1,p}_0(\Omega,X)$ via traces. Using standard convolution arguments Theorem \ref{extension theorem} yields

\begin{corollary}\label{Cdense}
Let $\Omega\subset\R^d$ be an open set with uniform Lipschitz boundary. Then for any $1\leq p<\infty$ the space $W^{1,p}(\Omega,X)\cap C(\overline{\Omega},X)$ is dense in $W^{1,p}(\Omega,X)$.
\end{corollary}

Using this, we can prove the Trace Theorem. On $\partial\Omega$ w consider the $(d-1)$-dimensional Hausdorff measure.
\begin{theorem}[Trace Theorem]
Let $\Omega\subset\R^d$ be an open set with uniform Lipschitz boundary, \mbox{$d\geq2$} and $1\leq p<\infty$. Then there exists a linear and continuous operator
\begin{align*}
\text{Tr}_X:W^{1,p}(\Omega,X)\rightarrow L^p(\partial\Omega,X)
\end{align*}
such that $\text{Tr}_Xu=u_{|\partial\Omega}$ for all $u\in W^{1,p}(\Omega,X)\cap C(\overline{\Omega},X)$.\\

In particular, given $u\in W^{1,p}(\Omega,X)$ we have that $u\in W^{1,p}_0(\Omega,X)$ if and only if $\text{Tr}_Xu=0$.
\end{theorem}

\begin{proof}
The case $X=\R$ is well known, see \cite[Theorems 15.10 \& 15.23]{Leoni}. For a function $u\in W^{1,p}(\Omega,X)\cap C(\overline{\Omega},X)$ we define \mbox{$\text{Tr}_Xu:=u_{|\partial\Omega}$}. This operator and the norm on $X$ commute in the sense that $\|\text{Tr}_Xu\|_X=\text{Tr}_\R\|u\|_X$. Hence by Corollary \ref{norm estimate} we have that
\begin{align*}
\|\text{Tr}_Xu\|_{L^p(\partial\Omega,X)}&=\|\,\|\text{Tr}_Xu\|_X\,\|_{L^p(\partial\Omega,\R)}\\
&=\|\text{Tr}_\R\|u\|_X\,\|_{L^p(\partial\Omega,\R)}\\
&\leq \|\text{Tr}_\R\|_\mathcal{L}\cdot\|\,\|u\|_X\,\|_{W^{1,p}(\Omega,\R)}\\
&\leq \|\text{Tr}_\R\|_\mathcal{L}\cdot\|u\|_{W^{1,p}(\Omega,X)},
\end{align*}
for any $u\in W^{1,p}(\Omega,X)\cap C(\overline{\Omega},X)$. By Corollary \ref{Cdense} we may extend $\text{Tr}_X$ to $W^{1,p}(\Omega,X)$. The continuity of the norm on $W^{1,p}(\Omega,X)$ implies that the operator still commutes with the norm as before, hence the last claim follows from Theorem \ref{W1p0 norm}.
\end{proof}

\section{Compact Resolvents via Aubin-Lions}\label{AubinLionsApplication}

As an application of our multidimensional Aubin-Lions Theorem we consider unbounded operators on $L^p(\Omega,H)$. Here $\Omega\subset\R^d$ is open and bounded and $H$ is a separable Hilbert space.\\

Let $B$ be a sectorial operator on $H$, that is $(-\infty,0)\subset\rho(B)$ and $\sup_{\lambda<0}\|\lambda(\lambda-B)^{-1}\|<\infty$. It follows from \cite[Proposition 3.3.8]{ArendtBatty} that $B$ is densely defined and $\lim_{\lambda\rightarrow-\infty}\lambda(\lambda-B)^{-1}x=x$ for all $x\in H$. For $1\leq p<\infty$ define $\tilde{B}$ on $L^p(\Omega,H)$ by
\begin{align*}
D(\tilde{B})=L^p(\Omega,D(B))\\
\tilde{B}u=B\circ u.
\end{align*}
Then $\tilde{B}$ is also sectorial. Now let $A$ be a sectorial operator on $L^p(\Omega,\R)$. We want to extend $A$ to a sectorial operator $\tilde{A}$ on $L^p(\Omega,H)$.

\begin{lemma}\label{Hilbert extension}
Let $T\in\mathcal{L}(L^p(\Omega,\R))$. Then there exists a unique bounded operator $\tilde{T}$ on $L^p(\Omega,H)$ such that $\tilde{T}(f\otimes x)=Tf\otimes x$ for all $f\in L^p(\Omega,\R)$ and $x\in H$. Moreover $\|\tilde{T}\|_{\mathcal{L}}=\|T\|_{\mathcal{L}}$.
\end{lemma}

\begin{proof}
See \cite[Theorem 2.9]{VanNeerven}.
\end{proof}

As a consequence of Lemma \ref{Hilbert extension}, given $\lambda<0$ there exists a unique bounded operator $\tilde{R}(\lambda)$ on $L^p(\Omega,H)$ such that $\tilde{R}(\lambda)(f\otimes x)=R(\lambda,A)f\otimes x$ for all $f\in L^p(\Omega,\R)$ and $x\in H$. It follows that $(\tilde{R}(\lambda))_{\lambda<0}$ is a pseudoresolvent on $(-\infty,0)$ and that $\lim_{\lambda\rightarrow-\infty}\lambda\tilde{R}(\lambda)u=u$ for all $u\in L^p(\Omega,H)$. Since $\textnormal{ker }\tilde{R}(\lambda)$ is independent of $\lambda$, it follows that $\tilde{R}(\lambda)$ is injective. Consequently there exists a unique operator $\tilde{A}$ on $L^p(\Omega,H)$ such that $(-\infty,0)\subset\rho(\tilde{A})$ and $\tilde{R}(\lambda)=R(\lambda,\tilde{A})$ for all $\lambda<0$. Thus $\tilde{A}$ is a sectorial operator on $L^p(\Omega,H)$. For tensors $u=f\otimes x$ we have
\begin{align*}
R(\lambda,\tilde{A})R(\lambda,\tilde{B})u=R(\lambda,A)f\otimes R(\lambda,B)x=R(\lambda,\tilde{B})R(\lambda,\tilde{A})u
\end{align*}
hence the two resolvents commute. If $\varphi_{\textnormal{sec}}(\tilde{A})+\varphi_{\textnormal{sec}}(\tilde{B})<\pi$, a result of DaPrato-Grisvard \cite[Section 4.2]{Arendt2004} says that $C=\tilde{A}+\tilde{B}$ is closable and $\overline{C}$ is a sectorial operator.\\

Assuming that $A$ and $B$ have compact resolvents, it is not obvious that also $\overline{C}$ has compact resolvent. We will show this if $A$ and $B$ enjoy maximal regularity and $D(A)\subset W^{1,p}(\Omega,\R)$.

\begin{lemma}
Assume that $D(A)\subset W^{1,p}(\Omega,\R)$. Then also $D(\tilde{A})\subset W^{1,p}(\Omega,H)$.
\end{lemma}

\begin{proof}
Let $\lambda<0$ and $j\in\{1,\ldots,d\}$. Then $D_j\circ R(\lambda,A)\in\mathcal{L}(L^p(\Omega,\R))$. By Lemma \ref{Hilbert extension} there exists $\tilde{Q}_j\in\mathcal{L}(L^p(\Omega,H))$ such that
\begin{align*}
\tilde{Q}_j(f\otimes x)=D_j(R(\lambda,A)f)\otimes x
\end{align*}
for all $f\in L^p(\Omega,\R)$ and $x\in H$. For every $u\in L^p(\Omega,H)$ there exist linear combinations $u_n$ of tensors such that $u_n\rightarrow u$ in $L^p(\Omega,H)$. Then $R(\lambda,\tilde{A})u_n\in W^{1,p}(\Omega,H)$ and $R(\lambda,\tilde{A})u_n\rightarrow R(\lambda,\tilde{A})u$ in $L^p(\Omega,H)$. Moreover
\begin{align*}
\|D_j R(\lambda,\tilde{A})u_m-D_jR(\lambda,\tilde{A})u_m\|_{L^p(\Omega,H)}\leq\|\tilde{Q}_j\|_\mathcal{L}\|u_m-u_n\|_{L^p(\Omega,H)}.
\end{align*}
Thus $R(\lambda,\tilde{A})u_n$ is a Cauchy sequence in $W^{1,p}(\Omega,H)$. This implies that $R(\lambda,\tilde{A})u\in W^{1,p}(\Omega,H)$.
\end{proof}

For the notion of \emph{bounded imaginary powers} we refer to \cite[Section 4.4]{Arendt2004} and the references given there.

\begin{theorem}
Let $A$ be a sectorial injective operator on $L^p(\Omega,\R)$ such that $D(A)\subset W^{1,p}(\Omega,\R)$, where $1<p<\infty$. Let $B$ be a sectorial injective operator on $H$ with compact resolvent. Suppose that both operators have bounded imaginary powers and that $\varphi_{\textnormal{bip}}(A)+\varphi_{\textnormal{bip}}(B)<\pi$. Then $\tilde{A}+\tilde{B}$ with domain $D(\tilde{A})\cap D(\tilde{B})$ is closed (and hence sectorial) and has compact resolvent.
\end{theorem}

\begin{proof}
It is easy to see that $\tilde{A}$ and $\tilde{B}$ both have bounded imaginary powers and $\varphi_{\textnormal{bip}}(\tilde{A})\leq \varphi_{\textnormal{bip}}(A)$, $\varphi_{\textnormal{bip}}(\tilde{B})\leq \varphi_{\textnormal{bip}}(B)$. It follows from the Dore-Venni Theorem, see e.g. \cite[Theorem 4.4.8]{Arendt2004}, that $\tilde{A}+\tilde{B}$ is a sectorial operator. Thus
\begin{align*}
D(\tilde{A}+\tilde{B})=D(\tilde{A})\cap D(\tilde{B})\subset W^{1,p}(\Omega,H)\cap L^p(\Omega,D(B))
\end{align*}
with continuous embedding by the Closed Graph Theorem. By Theorem \ref{AubinLions} the embedding $W^{1,p}(\Omega,H)\cap L^p(\Omega,D(B))\hookrightarrow L^p(\Omega,H)$ is compact. This implies that $\tilde{A}+\tilde{B}$ has compact resolvent.
\end{proof}

\begin{remark}
If $-A$ and $-B$ generate $C_0$-semigroups $(T(t))_{t\geq0}$ and $(S(t))_{t\geq0}$ of compact operators and if $p=2$, then another argument is possible. In fact, $-\tilde{A}$ and $-\tilde{B}$ generate the semigroups $(\tilde{T(t)})_{t\geq0}$ and $(\tilde{S(t)})_{t\geq0}$ which commute. Thus $U(t):=\tilde{T(t)}\tilde{S(t)}$ defines a $C_0$-semigroup on $L^2(\Omega,H)$ whose generator is $-\overline{C}$, the closure of $-\tilde{A}-\tilde{B}$. Kubrusly and Levan \cite[Theorem 2]{KubruslyLevan} proved that tensor products of compact operators on separable Hilbert spaces are compact. Thus $U(t)$ is compact and hence $-\overline{C}$ has compact resolvent. However, in our case this argument does not work, since there seems to be no simple formula relating the resolvents of $\tilde{A}$ and $\tilde{B}$ with that of $\overline{C}$.
\end{remark}

\newcommand{\etalchar}[1]{$^{#1}$}
\def\cprime{$'$}
\providecommand{\bysame}{\leavevmode\hbox to3em{\hrulefill}\thinspace}
\providecommand{\MR}{\relax\ifhmode\unskip\space\fi MR }
\providecommand{\MRhref}[2]{
  \href{http://www.ams.org/mathscinet-getitem?mr=#1}{#2}
}
\providecommand{\href}[2]{#2}

\end{document}